\documentclass[11pt,reqno]{amsart}
\usepackage{fullpage,amssymb,amsmath,mathrsfs,enumerate,tikz,url}

\usetikzlibrary{matrix}

\newtheorem{thm}{Theorem}

\newtheorem{lemma}[thm]{Lemma}

\newtheorem{question}{Question}
\newtheorem{case}{Case}

\theoremstyle{definition}

\newtheorem{remark}[thm]{Remark}

\theoremstyle{remark}
\newtheorem{claim}{Claim}

\def\Lip{\operatorname{Lip}}
\def\er{\mathbb R}

\def\setsep{:\;}
\def\en{\mathbb N}
\def\C{\mathcal{C}}
\def\F{\mathcal{F}}
\def\Norm{\|\cdot\|}
\def\Span{\mathrm{span}}
\def\dx{\, \mathrm{d}y}
\def\dy{\, \mathrm{d}x}

\begin{document}
\title{On the structure of Lipschitz-free spaces}
\author{Marek C\' uth}
\author{Michal Doucha}
\author{Przemys{\l}aw Wojtaszczyk}
\email{marek.cuth@gmail.com, m.doucha@post.cz, wojtaszczyk@icm.edu.pl}
\address[M.~C\' uth]{Charles University, Faculty of Mathematics and Physics, Department of Mathematical Analysis, Sokolovsk\'a 83, 186 75 Prague 8, Czech Republic}
\address[M.~C\' uth, M.~Doucha]{Instytut Matematyczny Polskiej Akademii Nauk, \' Sniadeckich 8, 00-656 Warszawa, Poland}
\address[P.~Wojtaszczyk]{{Interdisciplinary Centre for Mathematical and Computational Modelling, University of Warsaw, ul. Prosta 69, 02-838 Warszawa, Poland}}

\begin{abstract}In this note we study the structure of Lipschitz-free Banach spaces. We show that every Lipschitz-free Banach space over an infinite metric space contains a complemented copy of $\ell_1$. This result has many consequences for the structure of Lipschitz-free Banach spaces. Moreover, we give an example of a countable compact metric space $K$ such that $\F(K)$ is not isomorphic to a subspace of $L_1$ and we show that whenever $M$ is a subset of $\er^n$, then $\F(M)$ is weakly sequentially complete; in particular, $c_0$ does not embed into $\F(M)$.
\end{abstract}
\maketitle

\section*{Introduction}

Given a metric space $M$, it is possible to construct a Banach space $\F(M)$ in such a way that the Lipschitz structure of $M$ corresponds to the linear structure of $\F(M)$. This space $\F(M)$ is sometimes called ``Lipschitz-free space". We refer to the next section for some more details concerning the construction and basic properties of those spaces. Although Lipschitz-free spaces over separable metric spaces are easy to define, their structure is poorly understood to this day. The study of the linear structure of Lipschitz-free spaces over metric spaces has become an active field of study, see e.g. \cite{aude, dutFer, fonWoj, god, gk, kaufmann-published, lanPer}. In the first part of this paper we prove the following general result.

\begin{thm}\label{t:structure}Let $M$ be an infinite metric space. For the Banach space $X = \F(M)$, we have
\begin{enumerate}[\upshape (i)]
\item\label{l1c} $\ell_1 \stackrel{c}{\hookrightarrow} X$, i.e., there is a complemented subspace of $X$ isomorphic to $\ell_1$.
\end{enumerate}
From this we get
\begin{enumerate}[\upshape (i)]\setcounter{enumi}{1}
	\item\label{notInCK} $X\not \stackrel{c}{\hookrightarrow} \C(K)$, i.e., $X$ is not isomorphic to a complemented subspace of a $\C(K)$ space.
	\item $X^*$ is not weakly sequentially complete; in particular, $X$ is not isomorphic to $L^1$-predual.
	\item $X$ is not isomorphic to the Gurari\u{\i} space.
	\item $X$ is a projectively universal separable Banach space, i.e., for any separable Banach space $Y$ there exists a bounded linear operator from $X$ onto $Y$.
\end{enumerate}
\end{thm}

It often happens that the Lipschitz-free space over a ``small enough'' space is isomorphic to $\ell_1$. For example, if $M\subset\er$ is a set of measure zero or if $M$ is a separable ultrametric space, then $\F(M)$ is isomorphic to $\ell_1$, see \cite{god} and \cite{cutDou}. By the result of A. Dalet \cite{aude}, $\F(K)$ is a dual space with MAP whenever $K$ is a countable compact metric space. Hence, one could conjecture that in this case $\F(K)$ is isomorphic to $\ell_1$. We give an example which shows that this is not the case.

\begin{thm}\label{t:example}There is a countable compact metric space $K$ such that $\F(K)\not\hookrightarrow L_1$, i.e., $\F(K)$ is not linearly isomorphic to a subspace of $L_1$. Moreover, $K$ is a convergent sequence, i.e., it has only one accumulation point.
\end{thm}

If $M$ contains a bi-Lipschitz copy of $c_0$, then $\F(M)$ is an isomorphically universal separable Banach space; i.e., $\F(M)$ contains an isomorphic copy of every separable Banach space (for more details we refer to Section \ref{s:c0}). Y. Dutrieux and V. Ferenczi in \cite{dutFer} asked for the converse. The answer to this question is in general negative, because it follows from the result of P. Kaufmann \cite[Corollary 3.3]{kaufmann-published} that $\F(c_0)$ is isomorphic to $\F(B_{c_0})$ (thus, it is a universal) and of course, since $B_{c_0}$ is bounded, $B_{c_0}$ does not contain a bi-Lipschitz copy of $c_0$. However, we can still ask the question in the setting of Banach spaces.

\begin{question}\label{q:1}
Let $X$ be a Banach space. Is $\F(X)$ universal if and only if $X$ contains a bi-Lipschitz copy of $c_0$?
\end{question}

The following result is a partial progress towards the answer to this question. Up to our knowledge, Question \ref{q:1} is left open.

\begin{thm}\label{t:c0}Let $M\subset \er^n$ be an arbitrary set. Then $\F(M)$ is weakly sequentially complete.\\
Consequently, $c_0\not \hookrightarrow \F(M)$; i.e., $c_0$ is not linearly isomorphic to a subspace of $\F(M)$.
\end{thm}

To the best of our knowledge, it was not even known whether there could be a Lipschitz-free space which neither embeds into $L_1$ nor is universal. The example given in Theorem \ref{t:example} is one such example (because, by the result of A. Dalet \cite{aude}, $\F(K)$ is a separable dual space and so it does not contain $c_0$). Another one is the space $\F([0,1]^n)$, see Theorems \ref{t:c0} and \ref{t:notEmbed}.

In the last section of this note we mention some open problems related to the structure of isomorphically universal Lispchitz-free Banach spaces.

The notation and terminology we use are relatively standard. If $X$ and $Y$ are Banach spaces, the symbol $Y\hookrightarrow X$ (resp. $Y\not \hookrightarrow X$) means that $Y$ is (resp. is not) linearly isomorphic to a subspace of $X$. If $(M,d)$ is a metric space, $x\in M$ and $r \geq 0$, we use $U(x,r)$ and $B(x,r)$ to denote respectively the open and closed ball, i.e., the set $\{y\in M\setsep d(x,y)< r\}$ and $\{y\in M\setsep d(x,y)\leq r\}$.

\section{Basic facts about Lipschitz-free spaces}\label{s:2}
Let $(M,d)$ be a metric space with a distinguished point denoted by $0$. Consider the space $\Lip_0(M)$ of all real-valued Lipschitz functions that map $0\in M$ to $0\in\er$. It has a vector space structure and one can define a norm $\Norm_{\Lip}$ on $\Lip_0(M)$, where for $f\in \Lip_0(M)$, $\|f\|_{\Lip}$ is the minimal Lipschitz constant, i.e., $\sup\{\frac{|f(x)-f(y)|}{d(x,y)}:x\neq y\in M\}$. Then $\big(\Lip_0(M),\Norm_{\Lip}\big)$ is a Banach space.

For any $x\in M$ denote by $\delta_x\in \Lip_0(M)^*$ the evaluation functional, i.e., $\delta_x(f)=f(x)$ for every $f\in \Lip_0(M)$. Denote by $\F(M)$ the closure of the linear span of $\{\delta_x:x\in M\}$ with the dual space norm denoted simply by $\Norm$. Observe that for any $x,y\in M$ we have $\|\delta_x-\delta_y\|=d(x,y)$.

This space is usually called Lipschitz-free Banach space (also Arens-Eells space) and it is uniquely characterized by the following universal property.\\

Let $X$ be a Banach space and suppose $L: M\to X$ is a Lipschitz map such that $L(0) = 0$. Then there exists a unique linear map $\widehat{L}:\F(M)\to X$ extending $L$, i.e., the following diagram commutes
\begin{center}\begin{tikzpicture}
  \matrix (m) [matrix of math nodes,row sep=3em,column sep=4em,minimum width=2em]
  {
    M & X \\
   \F(M) & X \\};
  \path[-stealth]
    (m-1-1) edge node [left] {$\delta_M$} (m-2-1)
            edge node [above] {$L$} (m-1-2)
    (m-2-1.east|-m-2-2) edge [dashed] node [above] {$\widehat{L}$} (m-2-2)
    (m-1-2) edge node [right] {$\mathrm{id}_X$} (m-2-2);    
\end{tikzpicture}\end{center}
and $\|\widehat{L}\| = \|L\|_{\Lip}$ where $\Norm_{\Lip}$ denotes the Lipschitz norm of $L$.\\

This fact is usually referred to as folklore. The proof is so simple that we include it here.\\[-8pt]

Fix a Banach space $X$ and a Lipschitz map $L:M\rightarrow X$ mapping $0$ to $0$. Extend linearly $L$ from $M$ onto $\Span\{\delta_x\setsep x\in M\}$ and denote this extension by $\widehat{L}$. We only need to check that $\|\widehat{L}\|_{\Lip}=\|L\|_{\Lip}$. Pick some $a\in\Span\{\delta_x\setsep x\in M\}$. Then $\|\widehat{L}(a)\|_X=f(\widehat{L}(a))$ for some $f\in B_{X^*}$. However, $f\circ L$ then belongs to $\Lip_0(M)$ and $\|f\circ L\|_{\Lip}\leq \|L\|_{\Lip}$. It follows that $\|a\|\|L\|_{\Lip}\geq \|\widehat{L}(a)\|_X$ which proves the claim. Then we can extend $\widehat{L}$ to $\F(M)$, the closure of $\Span\{\delta_x\setsep x\in M\}$.\\[-8pt]

Using this universal property of $\F(M)$, it is immediate that $\F(M)^* = \Lip_0(M)$. Indeed, it is enough to consider $X=\er$ in the universal property mentioned above.

Further, it is useful to observe that whenever $N$ is a subspace of a metric space $(M,d)$, then $\F(N)$ is linearly isometric to a subspace of $\F(M)$. Indeed, the isometry is determined by sending $\delta_x\in\F(N)$ to $\delta_x\in\F(M)$; in order to see it is an isometry it is enough to use the well-known fact that any $f\in\Lip_0(N)$ can be extended to $F\in\Lip_0(M)$ with $\|f\|_{\Lip} = \|F\|_{\Lip}$, e.g., by putting $F(x): = \inf\{f(n) + \|f\|_{\Lip}d(n,x)\setsep n\in N\}$, $x\in M$; see e.g. \cite[Lemma 7.39]{johHaj}. Using this observation together with the universal property of $\F(M)$ we see that the Lipschitz structure of $M$ corresponds to the linear structure of $\F(M)$. For example, if $N$ is bi-Lipschitz equivalent (resp. isometric) to a subset of $M$, then $\F(N)$ is linearly isomorphic (resp. linearly isometric) to a subspace of $\F(M)$, etc.\\[-8pt]

The last basic fact we would like to mention here is that it is possible to give an `internal' definition of the norm on $\F(M)$, i.e. by a formula which refers only to the metric on the metric space $M$. This is in contrast to the `external' definition given above which refers to the space $\Lip_0(M)$ in the computation of the norm. This is described e.g. in \cite{weaver}. The proof is not difficult and so we include it here as well.

Let us consider another norm, denoted by $\Norm_{KR}$, on $\Span\{\delta_x\setsep x\in M\}$ which is a variant of the so-called Kantorovich-Rubinstein metric, a concept that penetrated many areas of mathematics and computer science. Let us identify $\delta_0$ with $0\in\F(M)$. For $a\in\Span\{\delta_x\setsep x\in M\setminus\{0\}\}$ set
$$\|a\|_{KR}=\inf\{|\alpha_1|\cdot d(y_1,z_1)+\ldots+|\alpha_n|\cdot d(y_n,z_n)\setsep a=\alpha_1(\delta_{y_1}-\delta_{z_1})+\ldots+\alpha_n(\delta_{y_n}-\delta_{z_n})\}.$$
It is straightforward to check that $\Norm_{KR}$ is a seminorm. Moreover, it is the largest seminorm $\Norm'$ on $\Span\{\delta_x\setsep x\in M\}$ satisfying $\|\delta_x-\delta_y\|'\leq d(x,y)$ for every $x,y\in M$. Indeed, any seminorm $\Norm'$ with that property must satisfy the inequality $\|x\|'\leq |\alpha_1|\|\delta_{y_1}-\delta_{z_1}\|'+\ldots+|\alpha_n|\|\delta_{y_n}-\delta_{z_n}\|'$ when $x=\alpha_1(\delta_{y_1}-\delta_{z_1})+\ldots+\alpha_n(\delta_{y_n}-\delta_{z_n})$ which shows that $\|x\|'\leq \|x\|_{KR}$. Since the standard norm $\Norm$ on $\F(M)$ satisfies the condition, we get that $\Norm\leq\Norm_{KR}$ which implies that $\Norm_{KR}$ is actually a norm and that $\|\delta_x-\delta_y\|_{KR}=d(x,y)$ for every $x,y\in M$.

Consider now the identity mapping $L:M\rightarrow \overline{\Span\{\delta_x\setsep x\in M\}}^{\Norm_{KR}}$ sending $x$ to $\delta_x$. It is an isometric embedding. By the universality property of $\F(M)$, $L$ extends to $\widehat{L}:\F(M)\rightarrow \overline{\Span\{\delta_x\setsep x\in M\}}^{\Norm_{KR}}$ which is still $1$-Lipschitz. It follows that $\Norm_{KR}\leq\Norm$, so the norms $\Norm$ and $\Norm_{KR}$ are one and the same. This fact is often referred to as the Kantorovich duality.

\section{Embedding of $\ell_1$}

The purpose of this section is to prove Theorem \ref{t:structure}. It will be deduced from the fact that $\ell_\infty$ embeds in the dual of a Lipschitz-free space, i.e., into the space of Lipschitz functions. Let us note that we do not know whether $\ell_\infty$ embeds isometrically into $\Lip_0(M)$ for every infinite metric space $M$. The natural way of embedding $\ell_\infty$ into the space of Lipschitz functions is described in the Lemma below.

\begin{lemma}\label{l:seq}Let $(M,d)$ be a metric space, $K > 0$ and let $(x_n,y_n)_{n\in\en}$ be a sequence of pairs of points from $M$ satisfying the following three conditions.
\begin{enumerate}[\upshape (i)]
\item\label{neq} For every $n\in\en$, we have $x_n\neq y_n$.
\item\label{cond} For every $n,m\in\en$, we have $x_m\notin U(y_n,K\cdot d(y_n,x_n))$.
\item\label{lin} For every $n\neq m$, we have $U(y_n,K\cdot d(y_n,x_n))\cap U(y_m, K\cdot d(y_m,x_m))=\emptyset$.
\end{enumerate}
Then $\ell_\infty\hookrightarrow \Lip_0(M)$.
\end{lemma}
\begin{proof}We may without loss of generality assume that $0=x_1$ (because $\Lip_0(M)\ni f\mapsto f - f(x_1)$ is a linear isometry onto the space of Lipschitz functions $g$ with $g(x_1) = 0$). For every $n\in\en$ we define $f_n(x):=\max\{d(y_n,x_n) - \tfrac{d(y_n,x)}{K},0\}$, $x\in M$. Then $f_n\in\Lip(M)$. Moreover, it is easy to see that $\|f_n\|_{\Lip} \leq \tfrac{1}{K}$ and $\|f\|_\infty = d(y_n,x_n)$. By \eqref{cond}, we have $K\cdot d(x_n,y_n)\leq d(x_1,y_n)$; hence, $f_n(0) = 0$.

Notice that condition \eqref{lin} implies that if $f_n(x)\neq 0$, then for every $m\neq n$ we have $f_m(x) = 0$. For every $x\in M$, we denote by $n(x)$ the unique $n\in\en$ with $f_n(x)\neq 0$ if it exists; otherwise, we put $n(x): = 1$. Finally, we define $T:\ell_\infty\to\Lip_0(M)$ by
$$T(\alpha)(x) : = \alpha\big(n(x)\big)\cdot f_{n(x)}(x),\qquad \alpha = (\alpha(n))_{n\in\en}\in\ell_\infty.$$
First, we will show that $T$ is linear and $\|T\|\leq \tfrac{2}{K}$. It is easy to see that $T$ is linear; hence, it suffices to show that for $\alpha = (\alpha(n))_{n\in\en}\in\ell_\infty$ with $\|\alpha\| = 1$, we have $\|T(\alpha)\|_{\Lip}\leq \tfrac{2}{K}$. Fix $x, y\in M$. We need to show that $|T(\alpha)(x) - T(\alpha)(y)|\leq \tfrac{2}{K}d(x,y)$. If $n(x) = n(y)$, this is easy because $f_{n(x)}$ is $\tfrac{1}{K}$-Lipschitz. Hence, we may assume that $n(y)\neq n(x)$. Thus, we have $f_{n(x)}(y) = 0 = f_{n(y)}(x)$ and
$$|T(\alpha)(x) - T(\alpha)(y)|\leq f_{n(x)}(x) + f_{n(y)}(y) = |f_{n(x)}(x) - f_{n(x)}(y)| + |f_{n(y)}(y) - f_{n(y)}(x)|\leq \tfrac{2}{K}d(x,y).$$

In order to see that $T$ is an isomorphism, we will use condition \eqref{cond}. Fix $\alpha = (\alpha(n))_{n\in\en}\in\ell_\infty$ and $N\in\en$. By \eqref{cond}, for every $k\in \en$, we have $f_k(x_N) = 0$; hence $T(\alpha)(x_N) = 0$ and we have
$$|T(\alpha)(x_N) - T(\alpha)(y_N)| = |T(\alpha)(y_N)| = |\alpha(N)|f_N(y_N) = |\alpha(N)|d(x_N,y_N).$$
Therefore, $\|T(\alpha)\|_{\Lip}\geq |\alpha(N)|$ and, since $N$ was arbitrary, $\|T(\alpha)\|_{\Lip}\geq \|\alpha\|_\infty$.
\end{proof}

The following result is the main step towards the proof of Theorem \ref{t:structure}.

\begin{thm}\label{t:fin}Let $M$ be an infinite metric space. Then $\ell_\infty\hookrightarrow \Lip_0(M)$.
\end{thm}
\begin{proof}First, note that we may without generality assume that $M$ is complete, because otherwise we take the completion $N$ of $M$ and use the obvious fact that $\Lip_0(\overline{A})$ is linearly isometric to $\Lip_0(A)$ for every $A\subset N$; in particular, $\Lip_0(N)$ is linearly isometric to $\Lip_0(M)$.

Now, we will prove the statement considering several cases. In each of them we will find a sequence of pairs of points from $M$ satisfying the assumptions of Lemma \ref{l:seq}.
\begin{case}\label{1}$M$ is unbounded; i.e., for every $K > 0$ there are $x,y\in M$ with $d(x,y) > K$.\end{case}
\begin{proof}[Proof for Case \ref{1}]
Pick a sequence $(z_n)_{n=1}^\infty$ in $M$ such that, for every $n\in\en$ we have $d(z_{n+1},0) > 2d(z_n,0)$. Now, for each $n\in\en$, put $x_n:=z_{2n-1}$ and $y_n:=z_{2n}$. We will show that the sequence $(x_n,y_n)_{n\in\en}$ satisfies the assumptions of Lemma \ref{l:seq} with $K = \tfrac{1}{3}$.

Obviously, (i) is satisfied. Further, for $n<m$ we have
\begin{align}
 \label{eq:leq} d(z_n,z_m) \leq d(z_n,0) + d(z_m,0) < (1 + 2^{-(m-n)}) d(z_m,0),\\
 \label{eq:geq} d(z_n,z_m) \geq d(z_m,0) - d(z_n,0) > (1 - 2^{-(m-n)}) d(z_m,0).
\end{align}

Let us show that (ii) holds. We need to show that, for $n, m\in\en$, we have
\begin{equation}
 \label{eq:mustii} d(z_{2m-1},z_{2n})\geq \tfrac{1}{3}\cdot d(z_{2n-1},z_{2n}).
\end{equation}
This is obvious if $m=n$. If $m<n$, then we have
$$d(z_{2m-1},z_{2n})\stackrel{\eqref{eq:geq}}{>}(1 - 2^{-(2n-2m+1)}) d(z_{2n},0)\stackrel{\eqref{eq:leq}}{>}(1 - 2^{-3})(1+2^{-1})^{-1}d(z_{2n},z_{2n-1}).$$
If $n < m$, then we have
$$d(z_{2m-1},z_{2n})\stackrel{\eqref{eq:geq}}{>}(1-2^{-1}) d(z_{2m-1},0)>(1-2^{-1})d(z_{2n},0)\stackrel{\eqref{eq:leq}}{>}(1-2^{-1})(1+2^{-1})^{-1}d(z_{2n},z_{2n-1}).$$
This proves \eqref{eq:mustii}; hence, condition (ii) from Lemma \ref{l:seq} is satisfied.

Finally, in order to see that (iii) holds, it is sufficient to see that, for $n\neq m$, we have
\begin{equation}
 \label{eq:mustiii} d(z_{2m},z_{2n}) > \tfrac{1}{3}\cdot \big(d(z_{2n-1},z_{2n}) + d(z_{2m-1},z_{2m})\big).
\end{equation}
We may assume that $n<m$ and then we have
\begin{equation*}
\begin{split}
d(z_{2n-1},z_{2n}) + d(z_{2m-1},z_{2m}) & \stackrel{\eqref{eq:leq}}{<} (1+2^{-1})^2 d(z_{2m},0) \stackrel{\eqref{eq:geq}}{<} (1+2^{-1})^2 (1-2^{-(2m-2n)})^{-1}d(z_{2m},z_{2n})\\
& \leq (1+2^{-1})^2 (1-2^{-2})^{-1}d(z_{2m},z_{2n}) = 3 d(z_{2m},z_{2n}).
\end{split}
\end{equation*}
This proves \eqref{eq:mustiii}; hence, condition (iii) from Lemma \ref{l:seq} is satisfied.
\end{proof}
\begin{case}\label{2}$M$ is bounded and there is a closed infinite subset $N\subset M$ such that each point $n\in N$ is isolated in $N$.\end{case}
\begin{proof}[Proof for Case \ref{2}]
Fix $N$ as above. Since $M$ is bounded, there is $D > 0$ such that, for every $x,y\in M$, we have $d(x,y)\leq D$. Since $N$ does not contain any nontrivial Cauchy sequence, there is an infinite subset $P\subset N$ which is uniformly discrete; i.e., there is $C > 0$ such that, for every $x,y\in P$ with $x\neq y$, we have $d(x,y)\geq C$ (this is a simple exercise using the classical Ramsey theorem, see e.g. \cite[Excercise 5.5]{prot}). Fix a one-to-one sequence $(a_n)_{n=0}^\infty$ of points from $P$ and, for every $n\in\en$, put $y_n = a_n$ and $x_n = a_0$. It remains to verify that the sequence $(x_n,y_n)_{n\in\en}$ satisfies the assumptions of Lemma \ref{l:seq} with $K = \min\{\tfrac{C}{2D},1\}$. It is clear that conditions \eqref{neq} and \eqref{cond} are satisfied, because $K \leq 1$. Moreover, for every $n\in\en$, we have $K\cdot d(x_n,y_n) \leq KD \leq C/2$; hence, $U(y_n,K\cdot d(x_n,y_n))\subset U(y_n,C/2)$ and, since $(y_n)_{n\in\en}$ is $C$-discrete, the balls are pairwise disjoint. This verifies condition \eqref{lin} from Lemma \ref{l:seq}.
\end{proof}
\begin{case}\label{3}$M$ is bounded and it contains infinitely many limit points.\end{case}
\begin{proof}[Proof for Case \ref{3}]
First, let us assume there is a sequence $(y_n)_{n\in\en}$ consisting of limit points in $M$ with $y_n\to y$. Then, for each $n\in\en$, put $r_n : = \operatorname{dist}(y_n,\{y_m\setsep m\neq n\}) > 0$ and pick some $x_n\in U(y_n,r_n/2)$ with $x_n\neq y_n$. Then it is easy to see that the sequence $(x_n,y_n)_{n\in\en}$ satisfies the assumptions of Lemma \ref{l:seq} with $K = 1$.

Otherwise, the set $N$ consisting of all the limit points in $M$ satisfies the assumptions of Case \ref{2}.
\end{proof}
Note that now it remains to handle the case when $M$ is compact and it contains a nontrivial convergent sequence consisting of isolated points. Indeed, by the already proven Cases \ref{1}-\ref{3}, we may assume $M$ is bounded and contains only finitely many limit points. Then either $M$ is compact, or there is an infinite closed set of isolated points in $M$ and we may apply Case \ref{2}.
\begin{case}\label{4}$M$ is compact and it contains a nontrivial convergent sequence consisting of isolated points.\end{case}
\begin{proof}[Proof for Case \ref{4}]
Let $(a_n)_{n\in\en}$ be a nontrivial convergent sequence consisting of isolated points with the limit point $a$. It is easy to construct by induction a sequence $(x_n,y_n)_{n\in\en}$ of pairs of points from $M$ with
\begin{itemize}
	\item[(a)] $\forall n\in\en:\quad y_n\in \{a_n:n\in\en\}$ and $y_n\notin\{y_m\setsep m<n\}\cup\{x_m\setsep m<n\}$,
	\item[(b)] $\forall n\in\en:\quad d(y_n,a) < \min\{d(y_n,y_m)\setsep m<n\}$ and
	\item[(c)] for every $n\in\en$, we pick $x_n$ to be any point with $d(x_n,y_n) = \operatorname{dist}(y_n,M\setminus\{y_n\})$.	
\end{itemize}
Now, having such a sequence $(x_n,y_n)_{n\in\en}$, it remains to check that it satisfies the assumptions of Lemma \ref{l:seq} with $K = 1$. Obviously, \eqref{neq} is satisfied. Moreover, we have $d(x_n,y_n)\leq d(y_n,x)$ for every $x\in M\setminus\{y_n\}$ and so in order to verify \eqref{cond}, it is enough to observe that, for $n,m\in\en$, we have $x_n\neq y_m$. This follows from (a) for $n < m$, from (c) for $n = m$ and from (b) for $n > m$, because in the last case we have $d(x_n,y_n)\leq d(y_n,a) < d(y_n,y_m)$.

It remains to verify \eqref{lin}. But this is easy, because, for every $n\in\en$, by the choice of $x_n$ we have $U(y_n,d(x_n,y_n)) = \{y_n\}$.
\end{proof}
Since the cases mentioned above cover all the possibilities, this completes the proof of Theorem \ref{t:fin}.
\end{proof}

\begin{proof}[Proof of Theorem \ref{t:structure}]This is a consequence of Theorem \ref{t:fin}. Indeed, it is a classical result that, for every Banach space $X$, $\ell_\infty\hookrightarrow X^*$ if and only if $\ell_1$ is isomorphic to a complemented subspace of $X$ \cite[Theorem 4]{besPe}; hence, \eqref{l1c} follows. Since any complemented subspace of a $\C(K)$ space contains $c_0$, see e.g. \cite[Theorem 5.1]{ros}, from \eqref{l1c} we get \eqref{notInCK} because $c_0$ is not isomorphic to a subspace of $\ell_1$. Since the dual space contains $c_0$, it is not weakly sequentially complete and so it is not isomorphic to $L^1(\mu)$ \cite[Corollary III.C.14]{woj}. Therefore, $X$ is not isomorphic to any $L^1$-predual; in particular, not to the Gurari\u{\i} space \cite{gur}; see also \cite[Theorem 2.17]{gar}. As it is well known that $\ell_1$ is projectively universal, i.e., for any separable Banach space $Y$ there exists a bounded linear operator from $\ell_1$ onto $Y$, the same is true for $X$ since $\ell_1$ is complemented there.
\end{proof}

\begin{remark}From the assumptions of Lemma \ref{l:seq} it is possible not only to deduce that $\ell_1$ is isomorphic to a complemented subspace of $\F(M)$, but it is even possible to describe relatively easily this subspace. Let us assume that $(x_n,y_n)_{n\in\en}$ is as in Lemma \ref{l:seq}. For each $n\in\en$, put $e_n:= \tfrac{\delta_{y_n} - \delta_{x_n}}{d(y_n,x_n)}\in\F(M)$. Then, using similar proof as in Lemma \ref{l:seq} we get that $(e_n)_{n\in\en}$ is $2/K$-equivalent to the $\ell_1$ basis. Moreover, consider functions $(f_n)_{n\in\en}$ from the proof of Lemma \ref{l:seq} and define $r:M\to \F(M)$ by $r(x):=\sum_{n\in\en} f_n(x)e_n$, $x\in M$. Then it is possible to verify that $r$ is a $2/K$-Lipschitz. Using the universal property of $r$ we find $P:\F(M)\to\F(M)$ with $P\circ\delta = r$ and $\|P\|\leq 2/K$. Finally, one can verify that $P$ is actually a projection onto $\overline{\operatorname{span}}\{e_n\setsep n\in\en\}$.
\end{remark}

\section{Embedding into $L_1$}

The purpose of this section is to prove Theorem \ref{t:example}. In order to prove it, we will need the following result. The proof is just a modification of the arguments from \cite{kis}.

\begin{thm}\label{t:notEmbed}For any measure $\mu$, $\F([0,1]^2)\not\hookrightarrow L_1(\mu)$.
\end{thm}
\begin{proof}In order to shorten our notation, put $I: = [0,1]^2$. If there is a measure $\mu$ with $\F(I)\hookrightarrow L_1(\mu)$, then there is a continuous linear mapping from $L_\infty(\mu)$ onto $\Lip_0(I)$. Since $L_\infty(\mu)$ is a commutative C*-algebra, there exists a compact Hausdorff space $K$ such that $L_\infty(\mu)$ is isometric to $\C(K)$. Hence, it suffices to show that there does not exist a bounded linear mapping $T: \C(K)\to \Lip_0(I)$ which is onto. We only show that the ``identity'' mapping $id:\Lip_0(I)\to W^{1,1}(I)$ is absolutely summing. Then the rest can be proved just copying line by line the arguments from \cite[Theorem 3]{kis}, where this statement is proved for the space $\C^1(I)$ instead of $\Lip_0(I)$ using the fact that `identity'' mapping $id:\C^1(I)\to W^{1,1}(I)$ is ``absolutely summing'' ($W^{1,1}(I)$ is the Sobolev space).  So consider the ``identity'' mapping $id:\Lip_0(I)\to W^{1,1}(I)$. More precisely, having a Lipschitz function $f$, we denote by $[f]$ the equivalence class containing all the functions which are equal to $f$ almost everywhere. The ``identity'' mapping is the mapping $f\mapsto [f]$. By the classical Rademacher's theorem, see e.g. \cite{nekZaj}, every Lipschitz function defined on $I$ is almost everywhere differentiable and so it is possible to put $\|[f]\|_W: = \int_{[0,1]^2} (|f(x,y)| + |\partial_1f(x,y)| + |\partial_2f(x,y)|)\dy \dx$. It is immediate that $\|[f]\|_W\leq 3\|f\|_{\Lip}$ and it remains to show that the mapping $f\mapsto [f]$ is absolutely summing; i.e., there is a constant $C$ such that whenever $(f_i)_{i=1}^m$ are functions from $\Lip_0(I)$, then
$$\sum_{i=1}^m \|[f_i]\|_W\leq C\sup\{\sum_{i=1}^m|x^*(f_i)|\setsep x^*\in \Lip_0(I)^ *, \|x^*\|\leq 1\}.$$
Let us define $\Phi:\Lip_0([0,1]^2)\to L_\infty(I)\oplus_1 L_\infty(I)\oplus_1 L_\infty(I)$ by $\Lip_0(I)\ni f\mapsto \Phi(f): = (f,\partial_1 f,\partial_2 f)$. Note that $\Phi$ is a linear bounded operator. Further, consider $\Psi :L_\infty(I)\oplus_1 L_\infty(I)\oplus_1 L_\infty(I)\to  L_1(I)\oplus_1 L_1(I)\oplus_1 L_1(I)$ defined as the identity. It is a standard fact, see e.g. \cite[Remark 8.2.9]{alKal}, that the identity operator from $L_\infty(I)$ to $L_1(I)$ is absolutely summing; hence, $\Psi$ is absolutely summing. It is a classical fact that composition of a bounded operator with an absolutely summing one is absolutely summing, see e.g. \cite[Proposition 8.2.5]{alKal}. Hence, $id = \Psi\circ \Phi$ is absolutely summing.
\end{proof}

\begin{remark}The result that $\F(\er^2)\not\hookrightarrow L_1$ is often mentioned as a result of A. Naor and G. Schechtmann \cite{naoSch}. The proof above shows that, using minor modifications, it actually follows already from \cite{kis}.
\end{remark}

The rest of this section is devoted to the proof of Theorem \ref{t:example}. First, we construct the countable compact space with one accumulation point and then in a series of claims we prove the statement.\\

For every $n\geq 2$, let $(A_n,d_n)$ be the set $\{(\frac{i}{n^2},\frac{j}{n^2})\setsep 0\leq i,j\leq n\}$ equipped with the Euclidean distance $d_n$ inherited from $\er^2$. Denote by $K$ the amalgamated metric sum of $A_n$'s over $0$. That is, we take $K$ to be the disjoint union $\coprod_n A_n$ with the zero element $(0,0)$ identified in all of them. The metric $d$ on $K$ is defined as follows. For $a,b\in K$ we set
$$d(a,b)=\begin{cases}
d_n(a,b) & \exists n (a,b\in A_n),\\
d_n(a,0)+d_m(b,0) & a\in A_n, b\in A_m, n\neq m.\\
\end{cases}$$

It is easy to check that $K$ is a countable compact metric space; in fact, it is a convergent sequence, i.e., it has only one accumulation point - the zero.
\begin{claim}\label{claimKaufmann}
$\F(K)$ is isometric to $\bigoplus_{\ell_1} \F(A_n)$.
\end{claim}
\begin{proof}
This is easy and proved e.g. in \cite[Proposition 5.1]{kaufmann-preprint}.
\end{proof}
For every $n$, consider the set $nA_n :=\{(\frac{i}{n},\frac{j}{n})\setsep 0\leq i,j\leq n\}$ again equipped with the Euclidean distance. Clearly, $\F(nA_n)$ is isometric to $\F(A_n)$. Indeed, since both spaces are finite-dimensional, it suffices to find an isometry of their duals; the mapping $\phi:\Lip_0(A_n)\to \Lip_0(nA_n)$ defined by $\phi(f)(x): = nf(\tfrac{x}{n})$, $x\in nA_n$, $f\in \Lip_0(A_n)$ is such an isometry. As a consequence we get the following.
\begin{claim}\label{claim2}
$\F(K)$ is linearly isometric to $\bigoplus_{\ell_1} \F(nA_n)$.
\end{claim}
Since $nA_n$, for each $n$, is a subset of $[0,1]^2$ we may and will consider $\F(nA_n)$ as a subspace of $\F([0,1]^2)$. Notice that $\bigcup_n nA_n$ is dense in $[0,1]^2$. We need one more technical claim which says that finite dimensional subspaces of $\F([0,1]^2)$ can be approximated by finite dimensional subspaces of $\F(nA_n)$ for large enough $n$. In the following, by $d_{BM}$ we denote the Banach-Mazur distance.
\begin{claim}\label{claim5}
Let $E\subseteq \F([0,1]^2)$ be a finite dimensional subspace and let $\varepsilon>0$ be arbitrary. Then there exist $n\in\en$ and a finite dimensional subspace $E'\subseteq \F(nA_n)$ such that $d_{BM}(E,E')<1+\varepsilon$.
\end{claim}
\begin{proof}
Let $e_1,\ldots,e_m$ be a basis of $E$. Since all norms on $E$ are equivalent, there is $D > 0$ such that for all $a\in\er^m$ we have $\sum_{i=1}^m |a(i)|\leq D\|\sum_{i=1}^m a(i) e_i\|$. Fix $\delta = \tfrac{\varepsilon}{2 + \varepsilon}$, i.e., such that $\tfrac{1+\delta}{1-\delta} = 1+\varepsilon$. Each $e_i$ can be $(\delta/2mD)$-approximated by some linear combination of elements from $\Span\{\delta_y:y\in [0,1]^2\}$. Without loss of generality, we may assume that for each $i\leq m$ such a linear combination is of the same length. So for each $i\leq m$, we choose some $\alpha_1^i \delta_{x_1^i}+\ldots+\alpha_l^i \delta_{x_l^i}\in \Span\{\delta_y\setsep y\in [0,1]^2\}$ such that $\|e_i-(\alpha_1^i \delta_{x_1^i}+\ldots+\alpha_l^i \delta_{x_l^i})\|<\delta/2mlD$.

Now, since $\bigcup_{n\in\en}nA_n$ is dense in $[0,1]^2$, if we take $n$ large enough then for every $i\leq m$ and $j\leq l$ we can find $a_j^i\in nAn$ such that $\|\delta_{x_j^i}-\delta_{a_j^i}\|<\frac{\delta}{2ml\alpha D}$, where $\alpha=\max\{|\alpha_j^i|\setsep j\leq l,i\leq m\}$. Consequently, for every $i\leq m$ we get
$$\|\alpha_1^i \delta_{x_1^i}+\ldots+\alpha_l^i \delta_{x_l^i}-(\alpha_1^i \delta_{a_1^i}+\ldots+\alpha_l^i \delta_{a_l^i})\|<\alpha\cdot l\cdot\frac{\delta}{2ml\alpha D}=\delta/2mD.$$
Thus, if for every $i\leq m$ we denote $\alpha_1^i \delta_{a_1^i}+\ldots+\alpha_l^i \delta_{a_l^i}$ by $e'_i$, we have $\|e_i-e'_i\|<\delta/mD$. Hence, for any $a\in\er^m$, we have $\|\sum_{i=1}^m a(i)e_i - \sum_{i=1}^m a(i)e'_i\| < \delta/D(\sum_{i=1}^m |a(i)|)\leq \delta \|\sum_{i=1}^m a(i)e_i\|$ and, consequently,
$$\|\sum_{i=1}^m a(i)e'_i\| < (1+\delta)\|\sum_{i=1}^m a(i)e_i\|\qquad\text{and}\qquad\|\sum_{i=1}^m a(i)e_i\| <\|\sum_{i=1}^m a(i)e'_i\| + \delta \|\sum_{i=1}^m a(i)e_i\|.$$
Denote by $E'$ the subspace $\Span\{e'_i:i\leq m\}\subseteq \F(nA_n)$. Using the above, the linear mapping determined by sending $e_i$ to $e'_i$, for $i\leq m$, is a witness of the fact that $d_{BM}(E,E')< \tfrac{1+\delta}{1-\delta} = 1+\varepsilon$.
\end{proof}

Let us now formulate a result of Lindenstrauss and Pe\l czy\' nski that will help us finish the proof.
\begin{thm}[Theorem 7.1 in \cite{linPel}]\label{thm:LinPel}
Let $X$ be a Banach space and fix $\lambda\geq 1$. If for every finite dimensional subspace $E$ of $X$ there exists a finite dimensional subspace $E'$ of $\ell_1$ such that $d_{BM}(E,E')\leq \lambda$, then there exists a measure $\mu$ and a subspace $Y$ of $L_1(\mu)$ such that $d_{BM}(X,Y)\leq \lambda$.
\end{thm}
We are now ready to finish the proof of Theorem \ref{t:example}. By Theorem \ref{t:notEmbed} we have that $\F([0,1]^2)$ does not embed into $L_1(\mu)$ for any measure $\mu$. However, then by Theorem \ref{thm:LinPel} we get that for every $N\in \en$ there exists a finite dimensional subspace $E_N$ of $\F([0,1]^2)$ such that for every finite dimensional subspace $F$ of $\ell_1$ we have $d_{BM}(E_N,F)>N$. Using Claim \ref{claim5}, for each $N$ we can find some $n(N)\in \en$ and finite dimensional subspace $E_{n(N)}$ of $\F(n(N)A_{n(N)})$ such that $d_{BM}(E_N,E_{n(N)})<2$.

Assume now that $\F(K)$ embeds into $L_1$ via some linear embedding of norm less than $N/8$ for some $N\in \en$. By Claim \ref{claim2}, $\oplus_{\ell_1} \F(nA_n)$ embeds into $L_1$ via some linear embedding $T$ of norm less than $N/8$. Now, $T$ restricted on $E_{n(N)}\subseteq \F(n(N)A_{n(N)})$ has still norm bounded by $N/8$. In particular, there is some finite dimensional subspace $Y_{n(N)}$ of $L_1$ such that $d_{BM}(E_{n(N)},Y_{n(N)})\leq N/8$. Since $L_1$ is finitely representable in $\ell_1$, see \cite[Proposition 11.1.7]{alKal}, there exists a finite dimensional subspace $Y_N$ of $\ell_1$ such that $d_{BM}(Y_N,Y_{n(N)})<2$.

Now putting all these inequalities together we get
$$\frac{N}{8}\geq d_{BM}(E_{n(N)},Y_{n(N)})\geq \frac{d_{BM}(E_N,Y_N)}{d_{BM}(E_N,E_{n(N)})\cdot d_{BM}(Y_N,Y_{n(N)})} > \frac{N}{4}$$
and that is a contradiction finishing the proof.

\begin{remark}\label{r:preprintVsPublished}The paper \cite{kaufmann-preprint} was published (see \cite{kaufmann-published}). However, in the published version the statement \cite[Proposition 5.1]{kaufmann-preprint}, which we cite in the proof of Claim \ref{claimKaufmann} above is missing. Thus, we would like to sketch the easy proof of it here.\\[5pt]
	Consider
	$$\Phi:\bigoplus_{\ell_\infty} \Lip_0(A_n)\to \Lip_0(K)$$
	defined by $\Phi\big((f_n)\big)(x) = f_n(x)$, $(f_n)\in \bigoplus_{\ell_\infty} \Lip_0(A_n)$, $x\in A_n$. Then it is easy to verify that $\Phi$ is an isometry onto and $w^*-w^*$ homeomorphism. Hence, it is the adjoint of an isometry from $\F(K)$ onto $\bigoplus_{\ell_1} \F(A_n)$.
	\end{remark}

	\begin{remark}\label{r:LanProch}It has been observed by G. Lancien and A. Proch\'azka that our method of proof actually gives that $K$ from the statement of Theorem \ref{t:example} can be taken as a subset of $[0,1]^2$ and that there does not exist a bi-Lipschitz embedding of $\F(K)$ into $L_1$. Let us sketch the argument here.\\[5pt]	
	First, the only place where we used the metric of $K$ was to prove Claim \ref{claimKaufmann}. However, it is easy to see that taking a sequence $(k_n)_{n\in\en}$ increasing fast enough, we have that $\F(\bigcup_{n\in\en}A_{k_n})$ is linearly isomorphic to $\bigoplus_{\ell_1} \F(A_{k_n})$ (using the same mapping $\Phi$ as in Remark \ref{r:preprintVsPublished}), which would be enough for the rest of the proof. Hence, we may have $K = \bigcup_{n\in\en}A_{k_n}$. Moreover, our proof gives that $\F(K)$ does not linearly embed into any Banach space finitely representable in $\ell_1$. If there was a bi-Lipschitz embedding of $\F(K)$ into $L_1$ then, by \cite[Corollary 7.10]{bl}$, \F(K)$ embeds linearly into $(L_1)^{**}$ which is by the principle of local reflexivity \cite[Theorem 11.2.4]{alKal} finitely represented in $L_1$ (in particular, $(L_1)^{**}$ is finitely represented in $\ell_1$ because $L_1$ is), a contradiction.
	\end{remark}

\section{Embedding of $c_0$}\label{s:c0}

Let $M$ be a separable metric space that contains a bi-Lipschitz copy of every separable metric space. By \cite[Theorems 2.12 and 3.1]{gk}, we have $X\hookrightarrow \F(X)$ for every separable Banach space $X$; therefore, $\F(M)$ is a universal separable Banach space, i.e., $\F(M)$ contains an isomorphic copy of every separable Banach space. Note that by the result of Aharoni \cite{aha} this is equivalent to the condition that $M$ contains a bi-Lipschitz copy of $c_0$. Y. Dutrieux and V. Ferenczi in \cite{dutFer} asked for the converse, see Question \ref{q:1}. In this section we prove Theorem \ref{t:c0}, making a partial progress towards the answer to this question.\\

Let $M$ be either $[0,1]^n$ or $\er^n$. By $\C^1(M)$ we denote the space of functions $F:M\to\er$ whose derivatives of order $\leq$1 are continuous on $M$. For $F\in\C^1(M)$ we define $\|F\|^1_\infty: = \max\{\|F\|_\infty, \|\partial_{x_i}F\|_\infty\setsep i\leq n\}$. It is well-known that the space $(\C^1([0,1]^n),\|\cdot\|^1_\infty)$ is a Banach space.

The following result was essentially proved by J. Bourgain \cite{bourg1}, \cite{bourg2}. The result of J. Bourgain concerns the space of smooth functions over $n$-dimensional torus; however, the same proof works for the $n$-dimensional cube. We refer also to \cite{woj}, where a more detailed proof of the result of J. Bourgain may be found (use Example III.D.30 and Theorem III.D.31  and conclude similarly as in the proof of Corollary III.C.14).
\begin{thm}\label{t:wscC1}For every $n\in\en$, the Banach space $(\C^1([0,1]^n))^*$ is weakly sequentially complete, i.e., weakly Cauchy sequences are weakly convergent.
\end{thm}

\begin{lemma}\label{l:rozsir}Let $A\subset \er^n$ be a finite set and $f:A\to \er$ a 1-Lipschitz function (on $\er^n$ we consider Euclidean norm). Then, for every $\varepsilon > 0$, there exists $g\in \C^1(\er^n)$, an extension of $f$ (i.e., $g\supset f$), with $\|g\|^1_\infty < \max\{\|f\|_\infty, 1\} + \varepsilon$.
\end{lemma}
\begin{proof}Find $\delta > 0$ such that the balls $\{B(a,2\delta)\setsep a\in A\}$ are pairwise disjoint. Fix some even Lipschitz function $\tau\in\C^1(\er)$ with $\tau(0) = 1$, $\|\tau\|_{\infty}\leq 1$ and $\{x\setsep \tau(x) \neq 0\}\subset (-\delta,\delta)$; e.g.
\[
\tau(x) = \begin{cases}
e^{-\tfrac{1}{\delta^2-x^2} + \tfrac{1}{\delta^2}} & |x| < \delta,\\
0 & \text{otherwise}.
\end{cases}
\]
Let $K$ be such that $\tau$ is $K$-Lipschitz and $K > 1$.

We may assume that $0\in A$ and $f(0) = 0$. First, we extend $f$ to a 1-Lipchitz function defined on $\er^n$; see e.g. \cite[Lemma 7.39]{johHaj}. We call this extension again $f$. Now, we find a 1-Lipschitz $\tilde{g}\in\C^1(\er^n)$ with $\|f-\tilde{g}\|_\infty < \varepsilon/2K$; e.g. using the standard integral convolution \cite[Lemma 7.1]{johHaj}.

For $a\in A$ define $\phi_a :\er^n\to\er$ by $\phi_a(x)= \big(f(a) - \tilde{g}(a)\big)\tau(\|x-a\|)$, $x\in \er^n$. Then $h:=\sum_{a\in A} \phi_a$ is a well-defined $\varepsilon/2$-Lipschitz function such that $\|h\|_\infty\leq \varepsilon/2K$ and $h(a) = f(a) - \tilde{g}(a)$ for every $a\in A$. Moreover, since on a Hilbert space the norm is smooth everywhere except 0 and since the function $\tau$ is even, it is easy to observe that $h\in\C^1(\er^n)$. It remains to put $g:=\tilde{g} + h$. Then we have $g\in\C^1(\er^n)$, $\|g\|_\infty < \|f\|_\infty + \varepsilon$ and $g$ is $(1+\varepsilon/2)$-Lipschitz; hence, $\|g\|^1_\infty < \max\{\|f\|_\infty, 1\} + \varepsilon$.
\end{proof}

\begin{remark}Note that it follows from Lemma \ref{l:rozsir} that whenever $A$ is a finite set in $[0,1]^n$ and $f$ a 1-Lipschitz function on $A$ with $f(0)=0$, there is $g\in\C^1(\er^n)$ with $\|g\|^1_\infty\leq \sqrt{n} + 1$. Therefore, by \cite[Theorem 1]{fef}, there is a linear extension operator $T:\Lip_0(A)\to \Lip_0(\er^n)$ with norm depending only on $n$; hence, $T^*|_{\F(\er^n)}$ is a projection from $\F(\er^n)$ onto $\F(A)$. Consequently, whenever we have $M\subset [0,1]^n$ and $A\subset M$ a finite set, $\F(A)$ is $C(n)$-complemented in $\F(M)$, where the constant $C(n)$ depends only on the dimension $n$. This gives another proof of the fact that $\F(M)$ has BAP whenever $M\subset [0,1]^n$ \cite[Proposition 2.3]{lanPer}.
\end{remark}

\begin{lemma}\label{l:freeSpaceSub}For every $n\in\en$, there is an isomorphism of $\F([0,1]^n)$ into $(\C^1([0,1]^n))^*$.
\end{lemma}
\begin{proof}Put $Y = \{f\in\C^1([0,1]^n)\setsep f(0) = 0\}$. Then $Y$ is a closed subspace of codimension 1; hence, it is complemented and $Y^*$ is isomorphic to a subspace of $(\C^1([0,1]^n))^*$. For every $x\in [0,1]^n$ we define $T(\delta_x)\in Y^*$ by $T(\delta_x)(f) := f(x)$, $f\in Y$. Extend $T$ linearly to the set $\Span\{\delta_x\setsep x\in[0,1]^n\}$. Now, it is enough to verify that $T$ is an isomorphism into $Y^*$.

Fix an element $\mu\in \Span\{\delta_x\setsep x\in[0,1]^n\}$. There are $k\in\en$, $\alpha\in\er^k$ and $x_1\ldots,x_k\in[0,1]^n$ with $\mu = \sum_{i=1}^k \alpha(i)\delta_{x_i}$. We have to find constants $C > 0$ and $D > 0$ with 
\[
\begin{split}
C \sup \left\{\left|\sum_{i=1}^k \alpha(i) f(x_i)\right|\setsep \|f\|^1_\infty \leq 1, f(0)=0\right\}& \leq \sup \left\{\left|\sum_{i=1}^k \alpha(i) f(x_i)\right|\setsep \|f\|_{\operatorname{Lip}}\leq 1, f(0) = 0\right\}\\
& \leq D \sup \left\{\left|\sum_{i=1}^k \alpha(i) f(x_i)\right|\setsep \|f\|^1_\infty\leq 1, f(0)=0\right\}.
\end{split}
\]
The existence of constant $C$ follows from the basic fact that every function with total differential bounded by $K$ is $K$-Lipschitz; see \cite[Theorem 9.19]{rudin}; hence, we may put $C = 1/\sqrt{n}$. The existence of $D$ follows from Lemma \ref{l:rozsir}, which gives $D = \sqrt{n}$.
\end{proof}

\begin{proof}[Proof of Theorem \ref{t:c0}]By \cite[Corollary 3.3]{kaufmann-published}, $\F(\er^n)$ is isomorphic to $\F([0,1]^n)$. Hence, $\F(\er^n)$ is weakly sequentially complete by Theorem \ref{t:wscC1} and Lemma \ref{l:freeSpaceSub}. Finally, using the fact mentioned in Section \ref{s:2} that $\F(M)$ is isometric to a subspace of $\F(\er^n)$, we see that $\F(M)$ is weakly sequentially complete. Consequently, $c_0$ does not embed isomorphically into $\F(M)$ because, as it is well known and easy to prove, $c_0$ is not weakly sequentially complete.
\end{proof}

\section{Open problems}

As it was mentioned in Section \ref{s:c0}, if $M$ contains a bi-Lipschitz copy of $c_0$, then $\F(M)$ is a universal separable Banach space. Hence, we have quite a rich family of universal separable Banach spaces. By Theorem \ref{t:structure}, they are all different from $\C(K)$ spaces and from the Gurari\u{\i} space. One example is Pe{\l}czy{\'n}ski's universal basis space $\mathbb{P}$ (which is unique up to isomorphism). This space is isomorphic to $\F(\mathbb{P})$, see \cite[p. 139]{gk}. Another example is the Holmes space, i.e., the Lipschitz-free space over the Urysohn universal metric space.
By \cite[Theorem 4.2]{fonWoj}, the Holmes space is not isomorphic to $\mathbb{P}$. By \cite[Theorem 5]{dutFer}, $\F(c_0)$ is isomorphic to each $\F(\C(K))$. It could be of some interest to find out what isomorphic types of universal Banach spaces  we are able to get using the Lipschitz-free construction. For example, the following seems to be open.

\begin{question}Is $\F(c_0)$ isomorphic to the Holmes space or to $\mathbb{P}$?
\end{question}

In the light of Theorem \ref{t:c0} it is also natural to ask the following.

\begin{question}Is it true that $c_0\hookrightarrow \F(\ell_2)$?
\end{question}

Note that $c_0$ does not bi-Lipschitz embed into $\ell_p$ ($1\leq p < \infty$), see e.g. \cite[p. 169]{bl}. Hence, the negative answer to the above question would be a partial progress towards the answer to Question \ref{q:1}. Similarly, we do not know the answer to the following question.

\begin{question}Is it true that $c_0\hookrightarrow \F(\ell_1)$?
\end{question}

%
%
%

\subsection*{Acknowledgements}
We wish to thank G. Lancien and A. Proch\'azka for the valuable observation contained in Remark \ref{r:LanProch}.
M. C\'uth was supported by Warsaw Center of Matemathics and Computer Science (KNOW--MNSzW). M. Doucha was supported by funds allocated to the implementation of the international co-funded project in the years 2014-2018, 3038/7.PR/2014/2, and by the EU grant PCOFUND-GA-2012-600415. P. Wojtaszczyk was supported by Polish NCN grant UMO-2011/03/B/ST1/04902.

\def\cprime{$'$} \def\cprime{$'$}

\end{document}